\documentclass[preprint,review,10pt]{elsarticle}
\usepackage{amsmath,amssymb,amsfonts,amsthm}
\usepackage{hyperref}
\usepackage{graphicx}
\usepackage{mathrsfs}

\usepackage{color}
\usepackage{cases}

\newtheorem{lemma}{Lemma}
\newtheorem{remark}{Remark}
\newtheorem{theorem}{Theorem}

\let\mr=\mathrm

\begin{document}

\begin{frontmatter}

%% Title, authors and addresses

%% use the tnoteref command within \title for footnotes;
%% use the tnotetext command for the associated footnote;
%% use the fnref command within \author or \address for footnotes;
%% use the fntext command for the associated footnote;
%% use the corref command within \author for corresponding author footnotes;
%% use the cortext command for the associated footnote;
%% use the ead command for the email address,
%% and the form \ead[url] for the home page:
%%
%% \title{Title\tnoteref{label1}}
%% \tnotetext[label1]{}
%% \author{Name\corref{cor1}\fnref{label2}}
%% \ead{email address}
%% \ead[url]{home page}
%% \fntext[label2]{}
%% \cortext[cor1]{}
%% \address{Address\fnref{label3}}
%% \fntext[label3]{}

\title{Optimal order of uniform convergence for  finite element method on  Bakhvalov-type  meshes\tnoteref{funding} }

\tnotetext[funding]{
%Jin Zhang is supported by National Science Foundation of China (11771257), Shandong Provincial Natural Science Foundation, China (ZR2017MA003) and  A Project of Shandong Province Higher Educational Science and Technology Program (J17KA169); Xiaowei Liu is supported by National Science Foundation of China (11601251), Shandong Provincial Natural Science Foundation, China (ZR2016AM13) and  A Project of Shandong Province Higher Educational Science and Technology Program (J16LI10).
This research is supported by National Science Foundation of China (11771257,11601251), Shandong Provincial Natural Science Foundation, China (ZR2017MA003).
}

\author[label1] {Jin Zhang\corref{cor1}}
\author[label2] {Xiaowei Liu \fnref{cor2}}
\cortext[cor1] {Corresponding author: jinzhangalex@hotmail.com }
\fntext[cor2] {Email: xwliuvivi@hotmail.com }
\address[label1]{School of Mathematics and Statistics, Shandong Normal University,
Jinan 250014, China}
\address[label2]{School of Mathematics and Statistics, Qilu University of Technology (Shandong Academy of Sciences), Jinan 250353, China}

\begin{abstract}
We propose a new  analysis of convergence for a $k$th  order  ($k\ge 1$)  finite element method, which is applied  on Bakhvalov-type  meshes  to a singularly perturbed two-point boundary value problem. A novel interpolant is introduced, which  has a simple structure and  is easy to generalize. By means of this interpolant, we prove an optimal order of  uniform convergence with respect to the perturbation parameter.  
%Furthermore, we obtain a supercloseness result on Bakhvalov-type  meshes for the first time.
 Numerical experiments illustrate these theoretical results.
\end{abstract}

\begin{keyword}
%% keywords here, in the form: keyword \sep keyword
Singular perturbation\sep Convection--diffusion equation \sep Finite element method \sep Bakhvalov mesh \sep Uniform convergence
%\sep Postprocessing
%% MSC codes here, in the form: \MSC code \sep code
%% or \MSC[2008] code \sep code (2000 is the default)
%\MSC[2010] 65N12 \sep 65N30 %\sep  65N50
\end{keyword}

\end{frontmatter}

%%
%% Start line numbering here if you want
%%
% \linenumbers

%% main text
%%%%%%%%%%%%%%%%%%%%%%%%%%%%%%%%%%%%%%%%%%%%%
%
%
%
%%%%%%%%%%%%%%%%%%%%%%%%%%%%%%%%%%%%%%%%%%%%%
\section{Introduction}
We consider the  two-point boundary value problem
\begin{equation}\label{eq:SPP-1d}
Lu:=-\varepsilon u''-b(x)u'+c(x)u=f(x)\quad \text{in $\Omega:=(0,1)$},\quad
u(0)=u(1)=0,
\end{equation}
where $\varepsilon$ is a positive parameter, $b$, $c$ and $f$ are sufficiently smooth functions such that  $b(x)\ge \beta>1$ on $\bar{\Omega}$ and  
\begin{equation}\label{eq:SPP-condition-1}
c(x)+\frac{1}{2}b'(x)\ge \gamma>0\quad \text{on $\bar{\Omega}$}
\end{equation}
with some constants $\beta$ and $\gamma$. The condition \eqref{eq:SPP-condition-1}  ensures that
the boundary value problem has a unique solution.
In the cases of interest the diffusion parameter $\varepsilon$
can be  arbitrarily small and satisfies $0<\varepsilon\ll 1$. Thus this problem is  \emph{singuarly perturbed} and its solution typically features a boundary layer of width $\mathcal{O}(\varepsilon\ln (1/\varepsilon))$ at $x = 0$ (see  \cite{Roo1Sty2Tob3:2008-Robust}).

Solutions to  singularly perturbed problems are characterized by the presence of boundary or interior \emph{layers}, where  solutions change  rapidly. 
Numerical solutions of these problems are of significant mathematical interest. 
Classical numerical methods are often inappropriate,   because  in practice  it is very unlikely that layers are fully resolved by  common  meshes.
Hence specialised numerical methods are designed to compute accurate approximate solutions in an efficient  way. For example,  standard numerical methods on layer-adapted meshes, which are fine in layer regions and standard outside,   are commonly used; see  \cite{Roo1Sty2Tob3:2008-Robust,Mil1Rio2Shi3:2012-Fitted} and  many references therein.   On these meshes, classical numerical methods    are  \emph{uniformly} convergent with respect to  the singular perturbation parameter; see \cite{Linb:2010-Layer}. 
 Among them, there are two kinds of popular grids: Bakhvalov-type meshes (B-meshes) and Shishkin-type meshes (S-meshes); see \cite{Linb:2010-Layer}. 
% Bakhvalov mesh \cite{Bakhvalov:1969-Towards} and Shishkin mesh \cite{Shishkin:1990-Grid} are two popular grids; also see \cite{Linb:2010-Layer} for Bakhvalov-type meshes (B-meshes) and Shishkin-type meshes (S-meshes). 

The accuracy of finite difference methods on these locally refined meshes has been extensively studied and sharp error estimations have been derived (see  \cite{Mil1Rio2Shi3:2012-Fitted,Kopteva:1999-the,Linb:2010-Layer}). 
For instance, in \cite{Linb:2010-Layer} the author presented convergence rates of
$\mathcal{O}(N^{-1})$ and $\mathcal{O}(N^{-1}\ln N)$
 for a first-order upwind difference scheme on Bakhvalov grid \cite{Bakhvalov:1969-Towards} and Shishkin grid \cite{Shishkin:1990-Grid}, respectively,  
where $N$ is the number of mesh intervals in each coordinate direction.  Usually, the performance of B-meshes is superior to that of S-meshes.  This advantage is more and more obvious when higher-order schemes are used.   Besides,  the width of the mesh subdomain used to resolve the layer is $\mathcal{O}(\varepsilon \ln (1/\varepsilon))$ for B-meshes and  $\mathcal{O}(\varepsilon \ln N)$ for S-meshes.  The former is independent of the mesh parameter $N$ and this property  will be important under certain circumstances.

For finite element methods, the development of numerical theories on B-meshes is completely different from one on S-meshes. 
On standard Shishkin meshes Stynes and O'Riordan \cite{Styn1ORior2:1997-uniformly} derived a  sharp  uniform  convergence in the energy norm for finite element method.   Henceforward  numerous articles deal with uniform convergence of finite element methods  on S-meshes; see e.g. \cite{Roo1Sty2Tob3:2008-Robust,Roos1Styn2:2015-Some,Zha1Liu2Yan3:2016-Optimal-EL,Zhan1Styn2:2017-Supercloseness,Liu1Sty2Zha3:2018-Supercloseness-CIP-EL,Teo1Brd2Fra3etc:2018-SDFEM,Russ1Styn2:2019-Balanced-norm} and the references therein. 
However, it is still open for the optimal  uniform convergence of finite element methods on B-meshes ( see   \cite[Question 4.1]{Roos1Styn2:2015-Some} for more details).

This dilemma  arises from the fact  that the standard Lagrange interpolant does not work for uniform convergence of finite element methods on B-meshes. More specifically, the Lagrange  interpolant  cannot provide enough stability in  $L^2$ norm on a special mesh interval, which lies in the fine part and  is adjacent to the coarse part of B-meshes.  In  \cite{Roos-2006-Error} and \cite{Brda1Zari2:2016-singularly} a  quasi-interpolant is used  and provides enough stability for  the optimal uniform convergence.  Unfortunately, in both articles the analysis  is limited to one dimension and linear finite  element.  It is hard  to extend  the analysis to higher dimensions  or higher-order finite  elements for singularly perturbed problems.

In this contribution we will study the optimal uniform convergence of 
a  $k$th  order  ($k\ge 1$)  finite element method  on Bakhvalov-type meshes.     A  novel interpolant is constructed by redefining the standard Lagrange interpolant to the solution.   This interpolant has a  simple structure and it can also be applied to higher-dimensional problems in a straightforward way.  By means of this novel function, we prove the optimal order of uniform convergence in a standard way.
%Furthermore, for the first time we prove a superclosenss result  on B-meshes; see Theorem \ref{the:main result}.   

The rest of the paper is organized as follows. In Section 2 we describe our regularity
on the solution $u$ to \eqref{eq:SPP-1d}, introduce two  Bakhvalov-type meshes and define the  finite element method.  Some  preliminary results for the subsequent analysis are also derived in this section.  In Section 3 we construct and analyze an interpolant $\Pi u$ for the uniform convergence on B-meshes. In Section 4 uniform convergence is obtained by means of the interpolant $\Pi u$ and careful derivations of the convective term in the bilinear form. 
%Also a supercloseness result is presented in this section. 
In Section 5, numerical
results illustrate our theoretical bounds.

 We use the standard Sobolev spaces $W^{m,p}(D)$, $H^m(D)=W^{m,2}(D)$, $H^m_0(D)$
% , $L^p(D)$ 
 for nonnegative
integers $m$ and $1\le p \le \infty$.  Here $D$ is any measurable subset of $\Omega$. We denote by $\vert \cdot \vert_{W^{m,p}(D)}$ and $\Vert \cdot \Vert_{W^{m,p}(D)}$ the semi-norms and the norms in $W^{m,p}(D)$, respectively. On $H^m(D)$, $\vert \cdot \vert_{m,D}$ and $\Vert \cdot \Vert_{m,D}$ are the usual Sobolev semi-norm and norm. Denote by $\Vert \cdot \Vert_{L^{p}(D)}$  the norms in the Lebesgue spaces $L^p(D)$. We use the notation $(\cdot,\cdot)_D$ and $\Vert \cdot \Vert_{D}$ for  the $L^2(D)$-inner product and the $L^2(D)$-norm, respectively. When $D=\Omega$ we drop the subscript $D$ from the  notation for simplicity.
Throughout the article, all constants $C$ and $C_i$ are independent of
$\varepsilon$ and the mesh parameter $N$; unsubscripted constants $C$ are generic and may take different values in different formulas while subscripted constants $C_i$ are fixed. 

%%%%%%%%%%%%%%%%%%%%%%%%%%%%%%%%%%%%%%%%%%%%%
%
%
%
%%%%%%%%%%%%%%%%%%%%%%%%%%%%%%%%%%%%%%%%%%%%%

\section{Regularity, Bakhvalov mesh and finite element method}\label{sec:regularity,mesh,method}
\subsection{Regularity of the solution }\label{sec:regularity}
Information about higher-order derivatives of  the solution $u$ of \eqref{eq:SPP-1d} are usually needed by uniform convergence  of finite element methods. Such estimations appeared in   \cite[Lemma 1.9]{Roo1Sty2Tob3:2008-Robust} and are reproduced in the following lemma.
\begin{lemma}\label{lem:regularity}
Let $k$ be some positive integer. Assume that \eqref{eq:SPP-condition-1} holds true and $b,c,f$ are sufficiently smooth. The solution $u$ of \eqref{eq:SPP-1d} can be decomposed  into
\begin{equation}\label{eq:decomposition}
u=S+E,
\end{equation}
where the smooth part $S$ and  the layer part $E$ satisfy $LS=f$ and $LE=0$, respectively. Furthermore, one has
\begin{equation}\label{eq:regularity}
|S^{(l)}(x)|\le C, \quad
|E^{(l)}(x)|\le C \varepsilon^{-l}\exp\left(-\frac{ x }{ \varepsilon }  \right)\quad \text{ for $0\le l\le k+1$}.
\end{equation}

\end{lemma}
Note  $k$ depends on the regularity of the coefficients, in particular
\eqref{eq:regularity} holds for any $q\in\mathbb{N}$ if $b,c,f\in C^{\infty}[0,1]$. 
%In this contribution, we require $q\ge 2$. 

\subsection{Bakhvalov mesh }
Bakhvalov mesh  first appeared in \cite{Bakhvalov:1969-Towards} and is constructed according to layer functions like $E$ in Lemma \ref{lem:regularity}. 
Its mesh generating function is piecewise and belongs to $C^1$.
Its breakpoint, which separates  the mesh generating function, must be solved by a nonlinear equation and  usually is not explicitly known ( see  \cite[Part I \S 2.4.1]{Roo1Sty2Tob3:2008-Robust}). 
 
In this article, we focus on two Bakhvalov-type  meshes  introduced in \cite{Roos-2006-Error} and \cite{Kopteva:1999-the,Kopt1Save2:2011-Pointwise}. Their breakpoints are known already, and both mesh generating functions do not belong to $C^1$ any longer. 
In \cite{Roos-2006-Error}  the Bakhvalov mesh is defined by
\begin{equation}\label{eq:Bakhvalov mesh-Roos}
x=\psi(t)=
\left\{
\begin{split}
& -\sigma \varepsilon \ln ( 1-2(1-\varepsilon)t ) \quad &&\text{for $t\in [0,1/2]$},\\
&1-d (1-t)\quad &&\text{for $t\in (1/2,1]$},
\end{split}
\right.
\end{equation}
where  $\sigma\ge k+1$ with  some positive integer  $k$   and $d$ is used to ensure the continuity of $\psi(t)$ at $t=1/2$.  The  mesh generating function in \cite{Kopteva:1999-the,Kopt1Save2:2011-Pointwise}   is defined by
\begin{equation}\label{eq:Bakhvalov mesh-Kopteva}
x=\varphi(t)=
\left\{
\begin{split}
&-\sigma \varepsilon  \ln(1-2t)  \quad &&\text{for $t\in [0,\vartheta]$},\\
&1-d_1 (1-t)\quad &&\text{for $t\in (\vartheta,1]$},
\end{split}
\right.
\end{equation}
where $\sigma\ge k+1$, $\vartheta=1/2-C_1\varepsilon$ with some positive constant $C_1$ independent of $\varepsilon$ and $N$,  $d_1$ is chosen so that $\varphi(t)$  is continuous at $t=\vartheta$. 
The original Bakhvalov mesh can be recovered from \eqref{eq:Bakhvalov mesh-Kopteva}
by setting $\vartheta=1/2-\mathcal{C}(\varepsilon)\varepsilon$, where 
\begin{equation}\label{eq:original-Bakhvalov-mesh}
0<C_2\le \mathcal{C}(\varepsilon) \le C_3.
\end{equation}
For technical reasons, we assume 
$C_1\le 1/(\varepsilon N )$ and therefore $1/2-N^{-1}\le \vartheta <1/2$.  We also assume  that $\varepsilon\le N^{-1}$ in our analysis, as is generally the case in practice. If $\varepsilon> N^{-1}$, one sets $\psi(t)=\varphi(t)=t$, which generate uniform meshes. 

Assume that $N/2$ is a positive integer and define the mesh points $x_i=\psi(i/N)$ or $x_i=\varphi(i/N)$ for $i=0,1,\ldots,N$. For both Bakhvalov meshes one usually  has $x_{N/2}\le 1/2$. Denote an arbitrary subinterval $[x_i,x_{i+1}]$ by $I_i$,  its length by $h_i=x_{i+1}-x_i$ and a generic subinterval by $I$.

\subsection{The finite element  method}
The weak form of problem \eqref{eq:SPP-1d} is to find $u\in H^1_0(\Omega)$ such that
\begin{equation}\label{eq:weak form-1d}
a(u,v)=(f,v)\quad \forall v\in H^1_0(\Omega),
\end{equation}
where 
$
a(u,v):=\varepsilon ( u', v')-(b u', v)
+(cu,v)
$. Note that the variational formulation \eqref{eq:weak form-1d} has a unique solution by means of the Lax-Milgram lemma.

 Define the $C^0$ finite element space on the Bakhvalov meshes
\begin{equation*}\label{eq:VN}
V^{N}=\{w\in C(\bar{\Omega}):\; w(0)=w(1)=0,\;
 \text{$w|_{I_i}\in P_k(I_i)$ for $i=0,\ldots,N-1$} \}.
\end{equation*}
The finite element method for \eqref{eq:weak form-1d} reads as
\begin{equation}\label{eq:FE-1d}
a(u^N,v^N)=(f,v^N)\quad \forall v^N\in V^N.
\end{equation}
The natural norm associated with $a(\cdot,\cdot)$ is defined by
\begin{equation*}
\Vert v \Vert_{ \varepsilon}:=
\left\{\varepsilon \vert  v \vert^2_1+ \Vert v \Vert^2\right\}^{1/2} \quad \forall v\in H^1(\Omega).
\end{equation*}
Using \eqref{eq:SPP-condition-1}, it is easy to see that   one has the coercivity
\begin{equation}\label{eq:coercivity}
a(v^N,v^N) \ge \alpha \Vert v^N \Vert_{\varepsilon}^2\quad \text{for all $v^N\in V^N$}
\end{equation}
with $\alpha=\min \{1,\gamma\}$. It follows that $u^N$ is well defined by \eqref{eq:FE-1d} (see \cite{Bren1Scot2:2008-mathematical} and references therein).

%%%%%%%%%%%%%%%%%%%%%%%%%%%%%%%%%%%%%%%%%%%%%
%
%
%
%%%%%%%%%%%%%%%%%%%%%%%%%%%%%%%%%%%%%%%%%%%%%
\subsection{Preliminary results of Bakhvalov meshes }\label{sec:convergence-2}
In this subsection, we present some important properties of the Bakhvalov meshes and the layer function $E$, which are necessary for our uniform convergence.

We present some properties about the step sizes of Bakhavlov meshes as follows.
\begin{lemma}\label{lem:Bakhvlov-mesh}
For  Bakhvalov mesh  \eqref{eq:Bakhvalov mesh-Roos}, one has
\begin{align}
&h_0\le h_1\le \ldots \le h_{N/2-2},\label{eq:mesh-1}\\
&\frac{1}{4}\sigma \varepsilon \le h_{N/2-2}\le \sigma \varepsilon,\label{eq:mesh-3}\\
&\frac{1}{2}\sigma \varepsilon \le h_{N/2-1}\le 2\sigma N^{-1},\label{eq:mesh-4}\\
&N^{-1}\le h_i\le 2N^{-1}\quad  N/2\le i \le N-1.\label{eq:mesh-5}
\end{align}
On Bakhvalov mesh  \eqref{eq:Bakhvalov mesh-Kopteva}, bounds   analogous to \eqref{eq:mesh-1}--\eqref{eq:mesh-5} also hold.
%Besides, for  Bakhvalov mesh  \eqref{eq:Bakhvalov mesh-Roos} and any fixed $\alpha\in (0,1]$ one has
%\begin{equation}\label{eq:mesh-6}
%h_{N/2-1}\le C\varepsilon^{1-\alpha}N^{-\alpha}.
%\end{equation}
\end{lemma}
\begin{proof}
We just consider Bakhvalov mesh \eqref{eq:Bakhvalov mesh-Roos} and the other mesh can be similarly analyzed. 
Recalling that  $x_{N/2}\le 1/2$ and the Bakhvalov mesh separates $[x_{N/2},1]$ into $N/2$ uniform subintervals, one obtains \eqref{eq:mesh-5}. For $0\le i \le N/2-1$,  one has
\begin{equation*}
h_i=x_{i+1}-x_i=\int_{i/N}^{(i+1)/N} \sigma \varepsilon \frac{2(1-\varepsilon)}{1-2(1-\varepsilon)t}\mr{d}t
\end{equation*}
and 
\begin{equation}\label{eq:h-in-layer-1}
\sigma\varepsilon\frac{2(1-\varepsilon)  }{1-2(1-\varepsilon)iN^{-1}} N^{-1} \le h_i\le   \sigma\varepsilon \frac{2(1-\varepsilon)   }{1-2(1-\varepsilon)(i+1)N^{-1}} N^{-1}. 
\end{equation}
From \eqref{eq:h-in-layer-1}, we can prove \eqref{eq:mesh-1}, \eqref{eq:mesh-3} and \eqref{eq:mesh-4} easily.  

%For any fixed $\alpha\in (0,1]$, standard arguments show 
%\begin{equation}\label{eq:N/2-1-I}
%\ln x\le\frac{x^{\alpha}}{\alpha}\quad \text{for $x\in [1,+\infty)$}. 
%\end{equation}
% From \eqref{eq:Bakhvalov mesh-Roos}  and \eqref{eq:N/2-1-I} we have
%\begin{align*}
%h_{N/2-1}=& \sigma \varepsilon \ln \frac{1-2(1-\varepsilon)(1/2-N^{-1})}{1-2(1-\varepsilon)1/2}=\sigma \varepsilon\ln \frac{\varepsilon+2(1-\varepsilon)N^{-1}}{\varepsilon}\\
%&\le \sigma \varepsilon \frac{1}{\alpha} \left( \frac{\varepsilon+2(1-\varepsilon)N^{-1}}{\varepsilon} \right)^{\alpha}\le  C\varepsilon^{1-\alpha} N^{-\alpha}.
%\end{align*}
%Thus \eqref{eq:mesh-6} is proved.
\end{proof}
%\begin{remark}
%Direct calculations show that the bound  \eqref{eq:mesh-6} does not hold true for Bakhvalov mesh \eqref{eq:Bakhvalov mesh-Kopteva} in general.
%\end{remark}

We collect some bounds of the layer function $E$ and the function $e^{-x/\varepsilon}$ on the Bakhvalov meshes in the following lemma. 
\begin{lemma}
On  Bakhvalov meshes  \eqref{eq:Bakhvalov mesh-Roos} and \eqref{eq:Bakhvalov mesh-Kopteva}, one has  
\begin{align}
&|E(x_{N/2-1})|\le C N^{-\sigma},\quad |E(x_{N/2})|\le C \varepsilon^{\sigma},\label{eq:layer-function-1}\\
&\Vert E \Vert_{ I_{N/2-1}}+\varepsilon\Vert E' \Vert_{ I_{N/2-1}}\le C\varepsilon^{1/2}N^{-\sigma}, \label{eq:layer-function-2}\\
&\Vert E' \Vert_{[x_{N/2},x_N]}\le C \varepsilon^{\sigma-1/2}. \label{eq:layer-function-22}
%\\
%&\sum_{i=N/2}^{N-1}N^{-1}e^{-2x_i/\varepsilon}\le C(\varepsilon+N^{-1})\varepsilon^{2\sigma}.\label{eq:layer-function-3}
\end{align}
For $0\le i\le N/2-2$  and $0\le \mu\le \sigma$, we have 
\begin{equation}\label{eq:layer-function-4} 
h_{i}^{ \mu}  \max\limits_{x_i\le x \le x_{i+1} }e^{- x/\varepsilon}=h_{i}^{ \mu} e^{- x_i/\varepsilon}\le C\varepsilon^{\mu} N^{-\mu}.
\end{equation}
\end{lemma}

\begin{proof}
We just consider Bakhvalov mesh \eqref{eq:Bakhvalov mesh-Roos} and the   mesh  \eqref{eq:Bakhvalov mesh-Kopteva} can be similarly analyzed. 

Recalling $\varepsilon\le N^{-1}$, we prove \eqref{eq:layer-function-1}, \eqref{eq:layer-function-2} and \eqref{eq:layer-function-22} directly from \eqref{eq:regularity}. 
%From \eqref{eq:mesh-5} one has 
%$$
%N^{-1}e^{-2x_i/\varepsilon}\le C \int_{x_{i-1}}^{x_i} e^{-2x/\varepsilon} \mr{d}x \quad \text{for $N/2+1\le i \le N$}
%$$
%and 
%\begin{align*}
%&\sum_{i=N/2}^{N-1}N^{-1}e^{-2x_i/\varepsilon}=N^{-1}e^{-2x_{N/2}/\varepsilon}+\sum_{i=N/2+1}^{N-1}N^{-1}e^{-2x_i/\varepsilon}\\
%\le &
%CN^{-1}\varepsilon^{2\sigma}+C \int_{x_{N/2}}^{x_N} e^{-2x/\varepsilon}\mr{d}x 
%\le  
%C(\varepsilon+N^{-1})\varepsilon^{2\sigma}.
%\end{align*}

Let $0\le i\le N/2-2$. 
From \eqref{eq:Bakhvalov mesh-Roos} one has
$$
-\sigma \varepsilon \ln (1-2(1-\varepsilon)i/N)=x_i\le x \le x_{i+1}=-\sigma \varepsilon \ln (1-2(1-\varepsilon)(i+1)/N),
$$
and for $x\in [x_{i},x_{i+1}]$
\begin{equation}\label{e-x-varepsilon}
(1-2(1-\varepsilon)(i+1)/N)^{\sigma}\le e^{- x/\varepsilon}\le e^{- x_i/\varepsilon}=(1-2(1-\varepsilon)i/N)^{\sigma}.
\end{equation}
From \eqref{eq:h-in-layer-1} and \eqref{e-x-varepsilon}, we have
\begin{equation*}
\begin{split}
&h_{i}^{\mu}\max\limits_{x_i\le x \le x_{i+1} }e^{-x/\varepsilon}
\le 
C^*_1 \varepsilon^{ \mu}N^{-\mu}  (1-2(1-\varepsilon)i/N)^{\sigma}(1-2(1-\varepsilon)(i+1)/N)^{-\mu}\\
\le &
C^*_1 \varepsilon^{ \mu} N^{-\mu}  ( 1-2(1-\varepsilon)i/N )^{\sigma-\mu}\left(\frac{ 1-2(1-\varepsilon)i/N }{ 1-2(1-\varepsilon)(i+1)/N } \right)^{\mu}\\
\le &
C^*_1C_2^*C^*_3 \varepsilon^{ \mu} N^{-\mu},
\end{split}
\end{equation*}
where $C^*_1=(2\sigma(1-\varepsilon))^{\mu}\le (2\sigma)^{\mu}$
 and  for $0\le i\le N/2-2$
 $$
C_2^*=( 1-2(1-\varepsilon)i/N )^{\sigma-\mu}\le 1,\; C_3^*=\left(\frac{ 1-2(1-\varepsilon)i/N }{ 1-2(1-\varepsilon)(i+1)/N } \right)^{\mu}
\le 2^{\mu}.
$$
Thus  \eqref{eq:layer-function-4} is proved.

\end{proof}
%%%%%%%%%%%%%%%%%%%%%%%%%%%%%%%%%%%%%%%%%%%%%
\section{Interpolation operator and interpolation errors}
Now  a new interpolation operator  is introduced, which is used for our uniform convergence. Set $x_{i+j/k}:=x_{i}+(j/k)h_i$ for $i=0,1,\ldots,N-1$ and $j=1,\ldots,k-1$. For any $v\in C^0(\bar{\Omega})$ its Lagrange interpolant $v^I\in V^{N}$ on each Bakhvalov mesh   is defined by 
$$
v^I=\sum_{i=0}^N v(x_i)\theta_i(x)+\sum_{i=0}^{N-1} \sum_{j=1}^{k-1}v(x_{i+j/k})\theta_{i+j/k}(x),
$$
where $\theta_i(x)$, $\theta_{i+j/k}(x)$ is the piecewise $k$th order Lagrange basis function satisfying the well-known delta properties associated with the nodes $x_i$ and $x_{i+j/k}$, respectively.  For the solution $u$ to \eqref{eq:SPP-1d}, recall \eqref{eq:decomposition} in Lemma \ref{lem:regularity} and define the interpolant $\Pi u$ %to $u$
 by
\begin{equation}\label{eq:interpolant-u-10-27}
\Pi u=S^I+\pi E,
\end{equation}
where $S^I$ is the Lagrange interpolant to $S$ and 
\begin{equation}\label{eq:interpolant-E-1}
(\pi E)(x)= \sum_{i=0,i\ne N/2-1}^{N}E(x_i)\theta_i(x)+ \sum_{i=0,i\ne N/2-1}^{N-1} \sum_{j=1}^{k-1}v(x_{i+j/k})\theta_{i+j/k}(x).
\end{equation}
Define 
\begin{equation}\label{eq:omega}
(\mathcal{P} E)(x)=E(x_{N/2-1})\theta_{N/2-1}(x)+\sum_{j=1}^{k-1}E(x_{N/2-1+j/k})\theta_{N/2-1+j/k}(x),
\end{equation}
and clearly  we have
\begin{align}
&(\pi E)(x)=E^I-(\mathcal{P} E)(x),\quad \Pi u=u^I-(\mathcal{P} E)(x),\label{eq:pi-E-1}\\
&\pi E|_{[x_0,x_{N/2-2}]\cup [x_{N/2},x_N] }=E^I|_{[x_0,x_{N/2-2}]\cup [x_{N/2},x_N] },\label{eq:pi-E-2}\\
& \Pi u\in V^N.\nonumber
\end{align}

Interpolation theories in Sobolev spaces \cite[Theorem 3.1.4]{Ciarlet:2002-finite} tell us that
\begin{equation}\label{eq:interpolation-theory}
\Vert v-v^I \Vert_{W^{l,q}(I_i)}\le C h_i^{k+1-l+1/q-1/p}\vert v \vert_{W^{k+1,p}(I_i)},
\end{equation}
for all $v\in W^{k+1,p}(I)$,  where $i=0,1,\ldots,N-1$, $l=0,1$ and $1\le p,q\le \infty$.

\begin{lemma}\label{lem:interpolation-error}
On   Bakhvalov meshes \eqref{eq:Bakhvalov mesh-Roos} and \eqref{eq:Bakhvalov mesh-Kopteva}, one has  
\begin{align}
&\Vert E-E^I \Vert_{L^{\infty}(\Omega)}+\Vert S-S^I  \Vert_{L^{\infty}(\Omega)}+\Vert u-u^I  \Vert_{L^{\infty}(\Omega)}\le 
CN^{-(k+1)},\label{eq:interpolation-error-1}\\
&\Vert E-E^I \Vert +\Vert S-S^I  \Vert+\Vert u-u^I  \Vert\le 
CN^{-(k+1)},\label{eq:interpolation-error-11}\\
&\Vert  E^I  \Vert_{I_{N/2-1} }\le Ch_{N/2-1}^{1/2}N^{-\sigma},\; \Vert  E^I  \Vert_{[x_{N/2},x_N] }\le C\varepsilon^{\sigma},
\label{eq:interpolation-error-2}\\
&\Vert E-E^I \Vert_{\varepsilon}+\Vert u-u^I \Vert_{\varepsilon}\le C N^{-k}, \label{eq:interpolation-error-3}\\
&\Vert (\mathcal{P} E)(x) \Vert_{\varepsilon} \le CN^{-\sigma},\label{eq:omega-energy-norm}
\end{align}
where $(\mathcal{P} E)(x)$ is defined in \eqref{eq:omega}.
\end{lemma}
\begin{proof}
We just consider Bakhvalov mesh \eqref{eq:Bakhvalov mesh-Roos} and   mesh  \eqref{eq:Bakhvalov mesh-Kopteva} can be similarly analyzed. 

From \eqref{eq:interpolation-theory} and \eqref{eq:regularity}, for $0\le i \le N/2-2$ one has
\begin{equation}\label{eq:E-EI-1}
\begin{split}
&\Vert E-E^I \Vert_{L^{\infty}(I_i)}\le  C h_i^{k+1} \vert E \vert_{W^{k+1,\infty}(I_i)}\\
\le & C\varepsilon^{-(k+1)}h_i^{k+1} e^{-x_i/\varepsilon}\le CN^{-(k+1)},
\end{split}
\end{equation}
where  we have used  \eqref{eq:layer-function-4} with $\mu=k+1$ and $\sigma\ge k+1$. For $N/2-1\le i \le N-1$ we have
\begin{equation}\label{eq:E-EI-2}
\Vert E-E^I \Vert_{L^{\infty}(I_i)}\le  \Vert E  \Vert_{L^{\infty}(I_i)}+\Vert  E^I \Vert_{L^{\infty}(I_i)}
\le Ce^{-x_i/\varepsilon}\le C N^{-\sigma}.
\end{equation}
Collecting \eqref{eq:E-EI-1}, \eqref{eq:E-EI-2} and noting $\sigma\ge k+1$, we prove $\Vert E-E^I \Vert_{L^{\infty}(\Omega)}\le C N^{-(k+1)}$. Lemma \ref{lem:Bakhvlov-mesh}, \eqref{eq:interpolation-theory} and \eqref{eq:regularity} yield
$\Vert S-S^I \Vert_{L^{\infty}(\Omega)}\le CN^{-(k+1)}$.   From \eqref{eq:decomposition} we prove \eqref{eq:interpolation-error-1}.
The bound \eqref{eq:interpolation-error-11} can be easily obtained from \eqref{eq:interpolation-error-1} and H\"{o}lder inequalities. 

From \eqref{eq:regularity} and direct calculations  one can easily prove \eqref{eq:interpolation-error-2}.

Now we are ready to analyze $\Vert E-E^I \Vert_{\varepsilon}$. 
First we  decompose $\varepsilon \Vert (E-E^I)' \Vert^2$ into the following two parts
\begin{equation}\label{eq:energy-IE-1}
\begin{split}
\varepsilon \Vert (E-E^I)' \Vert^2
=&\varepsilon \sum_{i=0}^{N/2-2}\Vert (E-E^I)'\Vert^2_{I_i} 
+\varepsilon \Vert  (E-E^I)'\Vert^2_{[x_{N/2-1},x_{N}]}\\
=:&\mathcal{S}_1+\mathcal{S}_2.
\end{split}
\end{equation}
From \eqref{eq:interpolation-theory}, \eqref{eq:regularity}, \eqref{eq:layer-function-4} with $\mu=(2k+1)/2$ and $\sigma\ge k+1$, we have
\begin{equation}\label{eq:energy-IE-2}
\begin{split}
 &\mathcal{S}_1\le C  \varepsilon \sum_{i=0}^{N/2-2} h_i^{2k} \vert E \vert_{k+1,I_i}^2
\le C \varepsilon  \sum_{i=0}^{N/2-2} h_i^{2k}\int_{ x_i }^{ x_{i+1} }\varepsilon^{-2(k+1)}e^{-2x/\varepsilon}\mr{d}x\\
\le &
C\varepsilon  \sum_{i=0}^{N/2-2}  h_i^{2k}\; \varepsilon^{-2(k+1)} e^{-2x_i/\varepsilon}h_i
\le
C\varepsilon  \sum_{i=0}^{N/2-2}  \varepsilon^{-2(k+1)} \left( h_i^{(2k+1)/2}e^{-x_i/\varepsilon} \right)^2\\
\le &
C\varepsilon  \sum_{i=0}^{N/2-2}  \varepsilon^{-2(k+1)}    \varepsilon^{2k+1} N^{-(2k+1)}\le C N^{-2k}.
\end{split}
\end{equation}
From a triangle inequality,  \eqref{eq:mesh-4}, \eqref{eq:mesh-5}, \eqref{eq:layer-function-2}, \eqref{eq:layer-function-22}, inverse inequality \cite[Theorem 3.2.6]{Ciarlet:2002-finite} and \eqref{eq:interpolation-error-2}, one has
\begin{equation}\label{eq:energy-IE-3}
\begin{split}
\mathcal{S}_2\le &C  \varepsilon \left( \Vert E' \Vert^2_{[x_{N/2-1},x_{N}]}+\Vert (E^I)' \Vert^2_{[x_{N/2-1},x_{N/2}]}+\Vert (E^I)' \Vert^2_{[x_{N/2},x_N]} \right)\\
\le &
C \varepsilon  (\varepsilon^{-1}N^{-2\sigma}+h^{-2}_{N/2-1} \Vert  E^I  \Vert^2_{[x_{N/2},x_N]} + N^2 \Vert  E^I  \Vert^2_{[x_{N/2},x_N]} )\\
\le &
C  N^{-2\sigma}+C\varepsilon^{2\sigma+1} N^2.
\end{split}
\end{equation}
Substituting \eqref{eq:energy-IE-2}, \eqref{eq:energy-IE-3} into \eqref{eq:energy-IE-1} and recalling $\varepsilon\le N^{-1}$ and $\sigma\ge k+1$, we obtain
$$
\varepsilon \Vert (E-E^I)' \Vert^2\le C N^{-2k}
$$
and prove $\Vert E-E^I\Vert_{\varepsilon}\le CN^{-k}$ from \eqref{eq:interpolation-error-11}.  From \eqref{eq:interpolation-theory} and Lemma \ref{lem:Bakhvlov-mesh}, one can easily prove $\Vert S-S^I\Vert_{\varepsilon}\le C(\varepsilon^{1/2} N^{-(k-1/2)}+N^{-(k+1/2)})$. A triangle inequality yields $\Vert u-u^I \Vert_{\varepsilon}\le CN^{-k}$.  Thus \eqref{eq:interpolation-error-3} is proved.

Now we consider \eqref{eq:omega-energy-norm}. Direct calculations yield  
\begin{align*}
\Vert (\mathcal{P} E)(x) \Vert_{\varepsilon}\le &|E(x_{N/2-1})| \Vert\theta_{N/2-1}(x) \Vert_{\varepsilon}+\sum_{j=1}^{k-1}|E(x_{N/2-1+j/k})| \Vert\theta_{N/2-1+j/k}(x) \Vert_{\varepsilon} \\
\le &
CN^{-\sigma} \sum_{j=0}^{k-1} \Vert\theta_{N/2-1+j/k}(x) \Vert_{\varepsilon}
\le  
C N^{-\sigma},
\end{align*}
where we have used \eqref{eq:omega}, \eqref{eq:layer-function-1}, \eqref{eq:mesh-3} and \eqref{eq:mesh-4}.
\end{proof}
%%%%%%%%%%%%%%%%%%%%%%%%%%%%%%%%%%%%%%%%%%%%%
\section{Uniform convergence}
Introduce $\chi:=\Pi u-u^N$.
From \eqref{eq:coercivity}, the Galerkin orthogonality, \eqref{eq:decomposition}, \eqref{eq:interpolant-u-10-27}, \eqref{eq:pi-E-1} and integration by parts for $\displaystyle\int_0^1b(\pi E-E)'\chi\mr{d}x$, one has
\begin{equation}\label{eq:uniform-convergence-1}
\begin{split}
&\alpha \Vert \chi \Vert_{\varepsilon}^2\le a(\chi,\chi)=a(\Pi u-u,\chi)\\
=&\varepsilon \int_0^1 (u^I-u)'\chi'\mr{d}x-\varepsilon \int_{0}^1(\mathcal{P} E)'\chi'\mr{d}x\\
&+\int_0^1b(\pi E-E) \chi'\mr{d}x-\int_0^1b(S^I-S)'\chi\mr{d}x\\
&+\int_0^1b' (\pi E-E)\chi\mr{d}x+\int_0^1c (u^I-u)\chi\mr{d}x-\int_0^1c (\mathcal{P} E) \chi\mr{d}x\\
&=:\mr{I}+\mr{II}+\mr{III}+\mr{IV}+\mr{V}+\mr{VI}+\mr{VII}.
\end{split}
\end{equation}
In the following we will analyze the terms in the right-hand side of \eqref{eq:uniform-convergence-1}. H\"{o}lder inequalities yield 
\begin{equation}\label{eq:I}
(\mr{I}+\mr{VI})+(\mr{II}+\mr{VII})\le C \Vert u-u^I \Vert_{\varepsilon}\Vert \chi \Vert_{\varepsilon}+ C \Vert \omega_E \Vert_{\varepsilon}\Vert \chi \Vert_{\varepsilon}
\le CN^{-k}\Vert \chi \Vert_{\varepsilon},
\end{equation}
where \eqref{eq:interpolation-error-3} and \eqref{eq:omega-energy-norm} have been used. From \eqref{eq:interpolation-theory} and \eqref{eq:regularity}, one has
$\Vert (S^I-S)'\Vert\le CN^{-k}$ and $\Vert \pi E-E \Vert\le CN^{-(k+1)}$ from \eqref{eq:interpolation-error-11} and \eqref{eq:omega-energy-norm}. Consequently we obtain
\begin{equation}\label{eq:IV}
 \mr{IV}+\mr{V} \le C (\Vert (S^I-S)'\Vert+\Vert \pi E-E \Vert)\Vert \chi \Vert    \le CN^{-k}\Vert \chi \Vert.
\end{equation}
 
We put the arguments for $\mr{III}$   in the following lemma.   
\begin{lemma}\label{lem:III}
 Let the mesh $\{x_i\}$ be either the Bakhvalov mesh \eqref{eq:Bakhvalov mesh-Roos} or the Bakhvalov mesh \eqref{eq:Bakhvalov mesh-Kopteva}. Let $\pi E$ be defined in \eqref{eq:interpolant-E-1}. Then one has
\begin{equation}\label{eq:lemma5}
|\mr{III}|=\left|\int_0^1b(\pi E-E)\chi'\mr{d}x\right|\le CN^{-k} \Vert \chi \Vert_{\varepsilon}.
\end{equation}

\end{lemma}
\begin{proof}
According to \eqref{eq:pi-E-2}, the term $(b(\pi E-E),\chi')$ is separated into three parts as follows:
\begin{equation}\label{eq:III-1}
\begin{split}
 \int_0^1b(\pi E-E)\chi'\mr{d}x 
=&\int_{x_0}^{x_{N/2-2}}b(E^I-E)\chi'\mr{d}x\\
&+\int_{x_{N/2-2}}^{x_{N/2}}b(\pi E-E)\chi'\mr{d}x
+\int_{x_{N/2}}^{x_{N}}b(E^I-E)\chi'\mr{d}x\\
=:&\mathcal{I}_1+\mathcal{I}_2+\mathcal{I}_3.
\end{split}
\end{equation}

From H\"{o}lder inequalities, \eqref{eq:interpolation-theory},  \eqref{eq:layer-function-4} with $\mu=k+1$ and $\sigma\ge k+1$, we obtain
\begin{equation}\label{eq:III-2}
\begin{split}
&|\mathcal{I}_1|\le C\sum_{i=0}^{N/2-3}\int_{x_i}^{ x_{i+1} }|E^I-E| |\chi'|\mr{d}x \\
\le &
C\sum_{i=0}^{N/2-3} \Vert E^I-E \Vert_{L^{\infty}(I_i)} \Vert \chi' \Vert_{L^{1}(I_i)}\\
\le &
C\sum_{i=0}^{N/2-3} h_i^{k+1} \varepsilon^{-(k+1)}e^{-x_i/\varepsilon} \cdot h^{1/2}_i\Vert \chi' \Vert_{ I_i } 
\le  
C\varepsilon^{1/2} \sum_{i=0}^{N/2-3} N^{-(k+1)}    \Vert \chi' \Vert_{ I_i }\\
\le &
C\varepsilon^{1/2} \left(\sum_{i=0}^{N/2-3} N^{-2(k+1)} \right)^{1/2}\left(\sum_{i=0}^{N/2-3} \Vert \chi' \Vert^2_{ I_i }\right)^{1/2}
\\
\le &
C N^{-(k+1/2)} \Vert \chi \Vert_{\varepsilon},
\end{split}
\end{equation}
where  \eqref{eq:mesh-1} and \eqref{eq:mesh-3} have been used.

From H\"{o}lder inequalities and inverse inequalities, one has
\begin{equation}\label{eq:III-3}
\begin{split}
&|\mathcal{I}_3|\le C \Vert E^I-E \Vert_{[x_{N/2},x_N]} \Vert \chi' \Vert_{[x_{N/2},x_N]} \\
\le &
C N^{-(k+1)}\cdot N  \Vert \chi \Vert_{[x_{N/2},x_N]} 
\le  
C N^{-k} \Vert \chi \Vert,
\end{split}
\end{equation}
where \eqref{eq:interpolation-error-11} has been used.

Now we analyze the term $\mathcal{I}_2$. Note  $\pi E=E^I-E(x_{N/2-1})\theta_{N/2-1}(x)$ on $[x_{N/2-2},x_{N/2-1}]$ and one has
\begin{equation}\label{eq:III-4}
\begin{split}
&\left|\int_{x_{N/2-2}}^{x_{N/2-1}} b(\pi E-E) \chi'\mr{d}x\right|\\
\le &
C\int_{x_{N/2-2}}^{x_{N/2-1}} |E^I-E| \; |\chi'|\mr{d}x+C|E(x_{N/2-1})|\int_{x_{N/2-2}}^{x_{N/2-1}} |\theta_{N/2-1}\chi'|\mr{d}x\\
\le &
C\left( \Vert E^I-E \Vert_{L^{\infty}(I_{N/2-2})}+|E(x_{N/2-1})| \right) \Vert \chi' \Vert_{L^1(I_{N/2-2})}\\
\le &
C(h_{N/2-2}^{k+1}\varepsilon^{-(k+1)}e^{-x_{N/2-2}/\varepsilon}+N^{-\sigma})\cdot h^{1/2}_{N/2-2} \Vert \chi' \Vert_{I_{N/2-2}}\\
\le &
C(N^{-(k+1)}+N^{-\sigma} )\Vert \chi \Vert_{\varepsilon, I_{N/2-2}}
\end{split}
\end{equation}
where \eqref{eq:interpolant-E-1}, H\"{o}lder inequalities, \eqref{eq:interpolation-theory}, \eqref{eq:layer-function-4} with $\mu=k+1$ and $\sigma\ge k+1$, \eqref{eq:mesh-3} have been used.  On $[x_{N/2-1},x_{N/2}]$, we have $\pi E=E(x_{N/2})\theta_{N/2}(x)$ from \eqref{eq:interpolant-E-1} and
\begin{equation}\label{eq:III-5}
\begin{split}
&\left|\int_{x_{N/2-1}}^{x_{N/2}} b(\pi E-E) \chi'\mr{d}x\right|\\
\le &
C|E(x_{N/2})|\int_{x_{N/2-1}}^{x_{N/2}} |\theta_{N/2}(x)|\; |\chi'|\mr{d}x+ C\int_{x_{N/2-1}}^{x_{N/2}} |E|\;|\chi'|\mr{d}x\\
\le &
C\left( \varepsilon^{\sigma} \Vert  \theta_{N/2} \Vert_{I_{N/2-1}}+C\Vert  E  \Vert_{I_{N/2-1}} \right) \Vert  \chi' \Vert_{I_{N/2-1}} \\
\le &
C\left( \varepsilon^{\sigma}  h_{N/2-1}^{1/2}+\varepsilon^{1/2}N^{-\sigma} \right) \Vert  \chi' \Vert_{I_{N/2-1}}\\
\le &
C(\varepsilon^{\sigma-1/2} N^{-1/2}+N^{-\sigma} )\Vert \chi \Vert_{\varepsilon, I_{N/2-1}},
\end{split}
\end{equation}
where H\"{o}lder inequalities, \eqref{eq:layer-function-1}, \eqref{eq:layer-function-2} and \eqref{eq:mesh-4} have been used.  From \eqref{eq:III-4} and \eqref{eq:III-5}  we prove 
\begin{equation}\label{eq:III-6}
|\mathcal{I}_2|\le CN^{-(k+1)}\Vert \chi \Vert_{\varepsilon},
\end{equation}
where $\varepsilon\le N^{-1}$ and $\sigma\ge k+1$ have been used.
Substituting \eqref{eq:III-2}, \eqref{eq:III-3} and \eqref{eq:III-6} into \eqref{eq:III-1}, we are done.
\end{proof}

Now we are in a position to present  the main result.  
\begin{theorem}\label{the:main result}
Let the mesh $\{x_i\}$ be either Bakhvalov mesh \eqref{eq:Bakhvalov mesh-Roos} or   Bakhvalov mesh \eqref{eq:Bakhvalov mesh-Kopteva} with $\sigma\ge k+1$. 
%Let $V^N$ be the piecewise $k$th order finite element space on these Bakhvalov meshes.  
Let $u$ and $u^N$ be the solutions of \eqref{eq:SPP-1d} and \eqref{eq:FE-1d}, respectively.  Then one has
\begin{align}
% \Vert  u^I-u^N \Vert_{\varepsilon}+\Vert \Pi u-u^N \Vert_{\varepsilon}\le & C N^{-k},\label{eq:uI-uN}\\
 \Vert u-u^N \Vert_{\varepsilon}\le & CN^{-k}.\label{eq:u-uN}
\end{align}
\end{theorem}
\begin{proof}
Substituting \eqref{eq:I}, \eqref{eq:IV} and  \eqref{eq:lemma5}  into \eqref{eq:uniform-convergence-1}, we obtain   $\Vert \Pi u-u^N \Vert_{\varepsilon}\le C N^{-k}$.
From  \eqref{eq:pi-E-1} and \eqref{eq:omega-energy-norm} we have $\Vert u^I-u^N \Vert_{\varepsilon}\le \Vert \Pi u-u^N \Vert_{\varepsilon}+\Vert (\mathcal{P} E) \Vert_{\varepsilon}\le C N^{-k}$.
% and prove \eqref{eq:uI-uN}.
From a triangle  inequality and  \eqref{eq:interpolation-error-3}, one has
$$
\Vert u-u^N \Vert_{\varepsilon}
\le   \Vert u-u^I \Vert_{\varepsilon}+\Vert u^I-u^N \Vert_{\varepsilon}\le  
 C N^{-k}.
$$
Thus we are done.
\end{proof}
\begin{remark}
For the original Bakhvalov mesh \cite{Bakhvalov:1969-Towards}, Theorem \ref{the:main result} also holds true because of the property \eqref{eq:original-Bakhvalov-mesh}.
The   analysis is similar to one on the mesh \eqref{eq:Bakhvalov mesh-Kopteva}. 
\end{remark}
\section{Numerical experiments}\label{sec:numerical experiments}
We now present the results of some numerical experiments in order to illustrate the conclusions of Theorem \ref{the:main result}, and to check if they are sharp. All calculations were carried out by using Intel Visual FORTRAN 11 and  the 
discrete problems were solved by the LU factorization.

The following boundary value problem is considered
\begin{equation}\label{eq:test}
\begin{split}
-\varepsilon u''-(3-x)u'+u=&f(x)\quad \text{in $\Omega=(0,1)$},\\
u(0)=u(1)=&0, 
\end{split}
\end{equation}
where the right-hand side  $f$ is chosen such that
\begin{equation}\label{eq:problem for numerical tests-1}
u(x)=(1-x)(1-e^{-2 x/\varepsilon})=1-x-e^{-2 x/\varepsilon}+xe^{-2 x/\varepsilon}.
\end{equation}
is the exact solution. The solution \eqref{eq:problem for numerical tests-1} exhibits typical boundary layer behavior.  

%{\color{red}介绍$k=1,2,3,4$,保证总的自由度大概4096左右;~
 For our numerical experiments we   consider $\varepsilon=10^{-4},10^{-5},\cdots,10^{-9}$, $k=1,2,3,4$ and $N=8,16,\cdots $. For both Bakhvalov meshes \eqref{eq:Bakhvalov mesh-Roos} and \eqref{eq:Bakhvalov mesh-Kopteva} we take $\sigma=k+1$. Set $C_1=5(k+1)/4$ in \eqref{eq:Bakhvalov mesh-Kopteva}.
%and $N=8,16,\ldots,2048$.}

We estimate the uniform errors for a fixed $N$ by taking the maximum error over a wide range of $\varepsilon$, namely
\begin{align*}
e^N:=&\max\limits_{\varepsilon=10^{-4},10^{-5},\ldots,10^{-9}}\Vert u-u^N \Vert_{\varepsilon}.
%\eta^N:=&\max\limits_{\varepsilon=10^{-4},10^{-5},\ldots,10^{-9}}\Vert u^I-u^N \Vert_{\varepsilon}.
\end{align*}
Rates of convergence $r^N_e$   are computed by means of the formula
$$
r^N_e=\log_2(e^N/e^{2N}).
% \quad\text{and} \quad r^N_{\eta}=\log_2(\eta^N/\eta^{2N}).
$$

The numerical results are presented in Tables \ref{table:1} and \ref{table:2}.  The errors $e^N$ and the convergence rates $r^N_e$ are in accordance with   Theorem \ref{the:main result} and illustrate its sharpness.  
%The errors $\eta^N$ and the convergence rates $r^N_{\eta}$ show that the bound \eqref{eq:uI-uN} in Theorem \ref{the:main result} is not sharp.
 Moreover, in Tables \ref{table:1} and \ref{table:2} we can observe that  Bakhvalov mesh \eqref{eq:Bakhvalov mesh-Roos} gives almost the same  performance as   Bakhvalov mesh \eqref{eq:Bakhvalov mesh-Kopteva}.

\begin{table}[!htbp]
\caption{Errors and convergence rates for problem \eqref{eq:test}   } 
\footnotesize
\begin{tabular*}{\textwidth}{@{\extracolsep{\fill}}c | cc | cc || cc | cc | }
\hline
  &\multicolumn{4}{c}{ $k=1$ }     &\multicolumn{4}{c}{ $k=2$ }\\
\hline
  &\multicolumn{2}{c}{B-mesh \eqref{eq:Bakhvalov mesh-Roos} }     &\multicolumn{2}{c}{ B-mesh \eqref{eq:Bakhvalov mesh-Kopteva} } &\multicolumn{2}{c}{B-mesh \eqref{eq:Bakhvalov mesh-Roos} }    &\multicolumn{2}{c}{ B-mesh \eqref{eq:Bakhvalov mesh-Kopteva} }\\
\hline
$N$   &  $e^N$    &  $r^N_e$  & $e^N$    &  $r^N_e$     &  $e^N$   & $r^N_e$ & $e^N$ & $r^N_e$    \\
\hline
8    &0.338E+00     &    1.02 &0.339E-01  &1.02  &0.103E+00  &2.00  & 0.103E+00    &    2.00\\
16     & 0.167E+00     &    1.00 &0.167E+00   &1.00&0.257E-01 &2.00   & 0.257E-01     &    2.00 \\
32    & 0.834E-01     &   1.00 &0.834E-01 &1.00 &0.642E-02 &2.00  & 0.642E-02    &    2.00 \\
64     & 0.417E-01     &    1.00  &0.417E-01  &1.00&0.160E-02  &2.00   & 0.160E-02    &    2.00 \\
128    & 0.208E-01    &    1.00  &0.208E-01  &1.00  &0.401E-03  &2.00   & 0.401E-03    &    2.00\\
256    & 0.104E-01     &    1.00 &0.104E-01 &1.00  &0.100E-03 &1.99   & 0.100E-03     &    1.99 \\
512     & 0.521E-02     &    1.00 &0.521E-02  &1.00 &0.251E-04  &2.00  & 0.251E-04   &    2.00 \\
1024   & 0.260E-02     &    1.00  &0.260E-02  &1.00  &0.627E-05  &2.00   & 0.627E-05     &    2.00\\
2048    & 0.130E-02 & --- &0.130E-02   &  --- &0.157E-05   &  ---      &0.157E-05   & --- \\
\hline
\end{tabular*}
\label{table:1}
\end{table}

\begin{table}[!htbp]
\caption{Errors and convergence rates for problem \eqref{eq:test}   } 
\footnotesize
\begin{tabular*}{\textwidth}{@{\extracolsep{\fill}}c | cc | cc || cc | cc | }
\hline
  &\multicolumn{4}{c}{ $k=3$ }     &\multicolumn{4}{c}{ $k=4$ }\\
\hline
  &\multicolumn{2}{c}{B-mesh \eqref{eq:Bakhvalov mesh-Roos} }     &\multicolumn{2}{c}{ B-mesh \eqref{eq:Bakhvalov mesh-Kopteva} } &\multicolumn{2}{c}{B-mesh \eqref{eq:Bakhvalov mesh-Roos} }    &\multicolumn{2}{c}{ B-mesh \eqref{eq:Bakhvalov mesh-Kopteva} }\\
\hline
$N$   &  $e^N$    &  $r^N_e$  & $e^N$    &  $r^N_e$     &  $e^N$   & $r^N_e$ & $e^N$ & $r^N_e$    \\
\hline
8    &0.301E-01     &    2.94 &0.301E-01  &2.94  &0.898E-02  &3.88  & 0.898E-02    &    3.88 \\
16     & 0.393E-02     &    2.99 &0.393E-02   &2.99&0.609E-03 &3.97   & 0.609E-03     &    3.97 \\
32    & 0.496E-03     &   3.00 &0.496E-03 &3.00 &0.388E-04 &3.99  & 0.388E-04    &    3.99 \\
64     & 0.622E-04     &    3.00  &0.622E-04  &3.00&0.244E-05  &4.00   & 0.244E-05    &    4.00 \\
128    & 0.778E-05    &    3.00  &0.778E-05  &3.00  &0.153E-06  &4.00   & 0.153E-06    &    4.00\\
256    & 0.973E-06     &    3.00 &0.973E-06 &3.00  &0.954E-08 &4.00  & 0.954E-08     &    4.00 \\
512     & 0.122E-06     &    3.00 &0.122E-06  &3.00 &0.596E-09  &3.84  & 0.596E-09   &    3.79 \\
1024   & 0.152E-07     &    ---  &0.152E-07  &---  &0.415E-10  &---   & 0.430E-10     &    ---\\
\hline
\end{tabular*}
\label{table:2}
\end{table}

%\section{Acknowledgements}
%The authors thank two unknown referees for some perceptive comments that led them to improve this paper.

%%%%%%%%%%%%%%%%%%%%%%%%%%%%%%%%%%%%%%%%%%%%%%%%%%%%%%%%%%%%%%%%%%%%%%%%%%%%%%%%%%%%%%%%%%%%%%%%%%%
%
%%%%%%%%%%%%%%%%%%%%%%%%%%%%%%%%%%%%%%%%%%%%%%%%%%%%%%%%%%%%%%%%%%%%%%%%%%%%%%%%%%%%%%%%%%%%%%%%%%
%% The Appendices part is started with the command \appendix;
%% appendix sections are then done as normal sections
%% \appendix

%% \section{}
%% \label{}

%% References
%%
%% Following citation commands can be used in the body text:
%% Usage of \cite is as follows:
%%   \cite{key}          ==>>  [#]
%%   \cite[chap. 2]{key} ==>>  [#, chap. 2]
%%   \citet{key}         ==>>  Author [#]

%% References with bibTeX database:

%\bibliographystyle{plain}
%%%\bibliographystyle{alpha}
%\bibliography{\myref/problem/LM-paper,\myref/method/FE-book,\myref/method/FE-paper,\myref/method/FD-paper,\myref/method/alex-paper,\myref/problem/singular-perturbation-problem}

%% Authors are advised to submit their bibtex database files. They are
%% requested to list a bibtex style file in the manuscript if they do
%% not want to use model1a-num-names.bst.

%% References without bibTeX database:

%%%%%%%%%%%%%%%%%%%%%%%%%%%%%%%%%%%%%%%%%%%%%%%%%%%%%%%%%%%%%%%%%%%%%%%%%%%%%%%%%%%%%%%%%%%%%%%%%%%%%%%%%%%%
%
%
%%%%%%%%%%%%%%%%%%%%%%%%%%%%%%%%%%%%%%%%%%%%%%%%%%%%%%%%%%%%%%%%%%%%%%%%%%%%%%%%%%%%%%%%%%%%%%%%%%%%%%%%%%%%%

\end{document}